\documentclass{amsart}
\usepackage{color}
\usepackage{amssymb}
\usepackage{amsmath}
\usepackage{amsthm}
\usepackage{amscd}
\usepackage{mathrsfs}
\usepackage{tikz}
\usepackage{bm}
\theoremstyle{definition}
\newtheorem{dfn}{Definition}[section]
\theoremstyle{plain}
\newtheorem{thm}[dfn]{Theorem}

\newtheorem{cor}[dfn]{Corollary}

\theoremstyle{remark}
\newtheorem{exam}[dfn]{Example}
\newtheorem{rem}[dfn]{Remark}
\title{Operator-valued Khintchine inequality for $\epsilon$-free semicircles}
\author{Beno\^{i}t Collins}
\address{Department of Mathematics, Kyoto University, Kitashirakawa Oiwake-cho, Sakyo-ku, 606-8502, Japan}
\email{collins@math.kyoto-u.ac.jp}
\author{Akihiro Miyagawa}
\address{Department of Mathematics, University of California, San Diego 9500 Gilman Drive \#0112 La Jolla, CA  92093-0112, USA}
\email{amiyagwa@ucsd.edu}

\begin{document}

\begin{abstract}
 We exhibit several bounds for operator norms of the sum of $\epsilon$-free semicircular random variables introduced in the paper of Speicher and Wysocza\'{n}ski. In particular, using the first and second largest eigenvalues of the adjacency matrix $\epsilon$, we show analogs of the operator-valued Khintchine-type inequality obtained by Haagerup and Pisier.      
\end{abstract}

\maketitle

\section{Introduction}
The mixture of classical and free independence has been considered for many years in different contexts. In group theory, Green \cite{Green} introduced the graph product of groups in 1990, which is the mixture of direct and free products associated with a given graph. After that, Caspers and Fima \cite{MR3626564} imported this graph product into the field of operator algebras. Independently, M\l otkowski \cite{MR2044286} studied the same object (he called it $\Lambda$-freeness) from the perspective of non-commutative probability theory, which was revisited by Speicher and Wysocza\'{n}ski \cite{MR3552222} (they called it $\epsilon$-freeness). In recent years, there have been some results in different directions. Charlesworth et.al. \cite{C24a,C24b} proved several fundamental properties (factoriality, amenability, fullness, etc.) of the graph product of von Neumann algebras and also constructed random permutation matrix models that asymptotically obey this graph independence. Magee and Thomas \cite{MT23} proved strong convergence of random matrix models for the graph independence, which is applied to the fundamental group of a closed hyperbolic manifold. There is also a result by C\'{e}bron et.al. \cite{COY24} about the central limit theorem for $\epsilon$-freeness such that the limit distribution can be characterized by graphon.

One of our motivations is the paper by the first-named author and his coauthors \cite{Collins_2023} where they study the sum of random matrices that are asymptotically $\epsilon$-free. In their paper, the authors consider this sum as random local Hamiltonians based on the paper by Charlesworth and the first author \cite{MR4248552}, which also appears in the context of quantum many-body system \cite{PhysRevB.100.245152}. Recently, Chen et al. \cite{CVV24} proved the strong convergence of their (Gaussian) random matrix model, which says that the operator norm of any polynomial in them converges to the operator norm of the same polynomial in $\epsilon$-free semicircles.   
In our paper, we show several bounds for the operator norm of the sum of $\epsilon$-free semicircle random variables. 
For upper bounds, we have the following result.
\begin{thm}\label{Main}
Let $s_1,\dots,s_d$ be a standard $\epsilon$-free semicircle. Then, for any $a_1,\ldots,a_d \in B(H)$, we have
\[\left \|\sum_{i=1}^d a_i \otimes s_i \right \|_{B(H \otimes \mathcal{F}_{\epsilon})}\le 2 \sqrt{\lambda_1+1} \max \left \{\left\|\sum_{i=1}^d a_i^*a_i\right\|_{B(H)}^{\frac{1}{2}}, \ \left\|\sum_{i=1}^d a_i a_i^*\right\|_{B(H)}^{\frac{1}{2}} \right \} \]
where $\lambda_1$ is the largest eigenvalue of the adjacency matrix of $\epsilon$. Moreover, when $\epsilon$ is a connected regular graph with a degree less than $d-1$, we also have
\[\left \|\sum_{i=1}^d a_i \otimes s_i \right \|_{B(H \otimes \mathcal{F}_{\epsilon})}\le 2  \sqrt{\frac{d(\lambda_2+1)}{d-(\lambda_1-\lambda_2)}}\max \left \{\left\|\sum_{i=1}^d a_i^*a_i\right\|_{B(H)}^{\frac{1}{2}}, \ \left\|\sum_{i=1}^d a_i a_i^*\right\|_{B(H)}^{\frac{1}{2}} \right \} \]
where $\lambda_2$ is the second largest eigenvalue of the adjacency matrix $\epsilon$.
\end{thm}
This result can be seen as
an analog
of the operator-valued Khintchine inequality observed by Haagerup and Pisier \cite{MR1240608} for the free group $\mathbb{F}_d$;
\[\left \|\sum_{|g|=1} a_g \otimes \lambda(g) \right \|_{B(H \otimes l^2(\mathbb{F}_d))}\le 2 \max \left \{\left\|\sum_{|g|=1}a_g^*a_g\right\|_{B(H)}^{\frac{1}{2}}, \ \left\|\sum_{|g|=1}a_g a_g^*\right\|_{B(H)}^{\frac{1}{2}} \right \} \]
where $a_g \in B(H)$. We note that Caspers, Klisse, and Larsen \cite{MR4156133} also proved a Khintchine-type inequality for the graph product. Although they use the information of cliques (that is, complete subgraphs), our inequality uses the eigenvalues of the adjacency matrix and is also optimal in several cases (see Remark \ref{OptimalCase}).

For lower bounds, we were unable to use the eigenvalues of the adjacency matrix $\epsilon$. Instead, we have a bound that involves the cliques in $\epsilon$.
\begin{thm}\label{CliqueEstimate}
Let $s_1,\ldots,s_d$ be $\epsilon$-free semicircular elements. Then we have
\[
\left\|\sum_{i=1}^d s_i\right\| \ge \max \left\{\sqrt{4\omega(\epsilon)^2+ d - \omega(\epsilon)},\ 2\sqrt{d} \right\}
\]
where $\omega(\epsilon)$ is the largest size of a clique in $\epsilon$. 
\end{thm}
In this theorem, the lower bound $\sqrt{4\omega(\epsilon)^2+ d - \omega(\epsilon)}$ is a better estimate than the trivial bound $2\sqrt{d}$ when $\omega(\epsilon)\ge \frac{1+\sqrt{1+48d}}{8}$. 

For the proofs, we use a representation of $\epsilon$-free semicircles on a kind of Fock spaces. This method has been applied to estimate operator norms of polynomials in $q$-Gaussian system \cite{MR1811255} and twisted-Gaussian system \cite{MR2180382}. We hope that our results will be a trigger for further studies of the analytic properties of graph-independent random variables.

\subsection*{Acknowledgement}
We are very grateful to Liang Zhao for multiple discussions during his visit to Kyoto, and in particular for explaining $\epsilon$-freenss to the second-named author. The authors thank Roland Speicher for the useful discussions and for telling us about the unpublished version of his paper. We also thank Zhiyuan Yang for telling us about the operator-valued Khintchine-type inequality. We are grateful to Patrick O. Santos for useful discussions on the paper \cite{santos2025khintchineinequalitiestracemonoids}.  
B. C. was supported by JSPS Grant-in-Aid Scientific Research (B) no. 21H00987, and Challenging Research (Exploratory) no. 23K17299.
A. M. was supported by JSPS Research Fellowships for Young Scientists, JSPS KAKENHI Grant Number JP 22J12186, JP 22KJ1817, and JSPS Overseas Research Fellowship.

\section{Preliminaries}

\subsection{$\epsilon$-independence}
Throughout the paper, $\epsilon$ denotes a simple (i.e., no multiple edges) undirected graph without self-loops. Let $d\in \mathbb{N}$ be the number of vertices in $\epsilon$, and we assign different numbers in $\{1,\ldots,d\}$ to the vertices in $\epsilon$. We also identify the graph $\epsilon$ with its adjacency matrix $\epsilon= (\epsilon_{ij})_{1\le i,j\le d}$, which is a symmetric matrix whose $(i,j)$ entries are $1$ if the $i$-th vertex is connected to the $j$-th vertex and $0$ otherwise. Note that diagonal entries of the adjacency matrix are $0$ since $\epsilon$ does not have self-loops. Based on the definition in \cite{MR3552222} by Speicher and Wysocza\'{n}ski, the $\epsilon$-freeness is defined as follows.
\begin{dfn}
For $n\in \mathbb{N}$, let $I_n^{\epsilon}$ be the set of elements $(i_1,\ldots,i_n) \in \{1,\ldots,d\}^n$ such that if there exists $1\le k<l \le n$ with  $i_k=i_l$, then there is a $k<m<l$ with 
$i_k\neq i_m$ and $\epsilon_{i_k,i_m}=0$. Let $(\mathcal{A},\phi)$ be a non-commutative probability space. Then, we say unital subalgebras $\{\mathcal{A}_i\}_{i=1}^d$ of $\mathcal{A}$ are $\epsilon$-free independent if they satisfy
\begin{itemize}
 \item If $\epsilon_{ij}=1$, then $\mathcal{A}_i$ and $\mathcal{A}_j$ commute.  

\item For each $n \in \mathbb{N}$, whenever we take $a_k \in \mathcal{A}_{i_k}$ ($k=1,\ldots,n)$ such that $\phi(a_k)=0$ and $(i_1,\ldots,i_n) \in I_n^{\epsilon}$, we have
\[\phi[a_1a_2,\cdots a_n]=0.\]
  
\end{itemize}
\end{dfn}

\subsection{Representation of $\epsilon$-free semicircles}\label{Representation}
We briefly discuss a representation of $\epsilon$-free semicircular random variables, which has been known by Speicher and Wysocza\'{n}ski \cite[Proposition 5.1 in the unpublished version]{MR3552222}, and Magee and Thomas \cite[Section 2.3.]{MT23}. In fact, Speicher and Wysocza\'{n}ski \cite[Proposition 5.1 in the unpublished version]{MR3552222} showed that $\epsilon$-free semicircles have the same joint (non-commutative) distributions as the mixed $(\epsilon_{ij})$-Gaussians which are represented as the sum of creation and annihilation operators $l_i+l_i^*$ satisfying the following mixed $(\epsilon_{ij})$-relation \cite{MR1213608,MR1289833};
\[l_i^*l_j-\epsilon_{ij}l_jl_i^*=\delta_{ij}I.\]
\begin{rem}
    In \cite[Theorem 29]{MR2180382}, Kr\'{o}lak shows a Haagerup-type inequality for the mixed $(q_{i,j})$-Gaussians with $\max|q_{i,j}|<1$ that estimates the operator norm of the sum of the generators. 
    However, we cannot directly apply this inequality to $(\epsilon_{ij})$-Gaussians since $\epsilon_{ij}$ is $0$ or $1$. Thus, we need to treat $\epsilon$-free semicircles separately.
\end{rem}
Here, we give an
alternative
construction of $\epsilon$-free semicircles that simplifies their manipulation. Let $\epsilon$ be given. We define an equivalence relation, denoted by $\sim_{\epsilon}$, on the set of words $[d]^*=\bigsqcup_{n=0}^\infty \{1,\ldots,d\}^n$ with the empty word $0$ (that is, $\{1,\ldots,d\}^0=\{0\}$) by
$$u i j v \sim_{\epsilon} u j i v \quad \text{if and only if} \quad \epsilon_{i j} =1$$
where $u,v \in [d]^*$.
Let $[d]_{\epsilon}^*$ denote the set of equivalence classes of $[d]^*$ with respect to $\sim_{\epsilon}$. Let us consider the following Hilbert space
$$\mathcal{F}_{\epsilon} = \bigoplus_{w \in [d]_{\epsilon}^*} \mathbb{C} e_w$$
where $\mathbb{C}e_w$ is a one-dimensional Hilbert space with unit vector $e_w$ and we see $e_0$ as the vacuum.
As in the full Fock space, we define left creation operators $\{l_i\}_{i=1}^d$ by
$$l_i e_w = e_{iw}.$$
Note that $l_i$'s are isometries on $\mathcal{F}_{\epsilon}$ since $i u \not\sim_{\epsilon} i v$ if $u \not\sim_{\epsilon} v$.
In this setting, the left annihilation operator $l_i^*$ is given by
$$l_i^* e_w = \begin{cases} e_{w'} \ \text{if} \ w\sim_{\epsilon} iw' \\ 0 \ \text{otherwise.}  \end{cases}$$
Note that since  we have $u\sim_{\epsilon}v$ if $i u \sim_{\epsilon}iv$, $l_i^*$ is well-defined and satisfies
$$\langle l_i^*e_u, e_v\rangle = \langle e_u, l_i e_v \rangle.$$
Moreover, one can directly check the $(\epsilon_{ij})$-relation
\[l_i^*l_j-\epsilon_{ij}l_jl_i^*=\delta_{ij}I.\]
Put $s_i = l_i + l_i^*$. One sees that each $s_i$ has the standard semicircle distribution with respect to the vacuum state.
This construction is a rephrasing of the construction in \cite{MR1213608,MR1289833} (with $q_{ij}=\epsilon_{ij}$) and \cite[Section 2.1]{MT23}
using an orthonormal basis $\{e_w\}_{w\in[d]^*_{\epsilon}}$. Under this setting, \cite[Proposition 5.1 in the unpublished version]{MR3552222} and \cite[Lemma 2.8]{MT23}
 can be reformulated as follows.
 \begin{thm}
In the setting above,
$\{s_i\}_{i=1}^d$ forms an $\epsilon$-free family of semicircles with respect to the vacuum state $\tau(\cdot)=\langle \ \cdot \ e_0,e_0\rangle$. 
\end{thm}

\section{Main results}
Throughout this section, let $H$ be a Hilbert space and $B(H)$ be the set of bounded operators on $H$. In addition, $\mathcal{F}_{\epsilon}$ denotes the Hilbert space described in the previous section. 

First of all, we prove an analog of the operator-valued Khintchine inequality observed by Haagerup and Pisier \cite{MR1240608} for the free group $\mathbb{F}_d$;
\[\left \|\sum_{|g|=1} a_g \otimes \lambda(g) \right \|_{B(H \otimes l^2(\mathbb{F}_d))}\le 2 \max \left \{\left\|\sum_{|g|=1}a_g^*a_g\right\|_{B(H)}^{\frac{1}{2}}, \ \left\|\sum_{|g|=1}a_g a_g^*\right\|_{B(H)}^{\frac{1}{2}} \right \} \]
where $a_g \in B(H)$.
\begin{thm}
Let $s_1,\dots,s_d$ be standard $\epsilon$-free semicircles. Then, for any $a_1,\ldots,a_d \in B(H)$, we have
\[\left \|\sum_{i=1}^d a_i \otimes s_i \right \|_{B(H \otimes \mathcal{F}_{\epsilon})}\le 2 \sqrt{\lambda_1+1} \max \left \{\left\|\sum_{i=1}^d a_i^*a_i\right\|_{B(H)}^{\frac{1}{2}}, \ \left\|\sum_{i=1}^d a_i a_i^*\right\|_{B(H)}^{\frac{1}{2}} \right \} \]
where $\lambda_1$ is the largest eigenvalue of the adjacency matrix $\epsilon$. Moreover, when $\epsilon$ is a connected regular graph with a degree less than $d-1$, we also have
\[\left \|\sum_{i=1}^d a_i \otimes s_i \right \|_{B(H \otimes \mathcal{F}_{\epsilon})}\le 2  \sqrt{\frac{d(\lambda_2+1)}{d-(\lambda_1-\lambda_2)}}\max \left \{\left\|\sum_{i=1}^d a_i^*a_i\right\|_{B(H)}^{\frac{1}{2}}, \ \left\|\sum_{i=1}^d a_i a_i^*\right\|_{B(H)}^{\frac{1}{2}} \right \} \]
where $\lambda_2$ is the second largest eigenvalue of the adjacency matrix $\epsilon$.
\end{thm}
\begin{proof}
We use the 
representation $s_i=l_i+l_i^*$ on $\mathcal{F}_{\epsilon}$
in section \ref{Representation}
where we have $l_i^*l_j-\epsilon_{ij} l_jl_i^* =\delta_{ij} I$.
Then, we have
\begin{align*}
\left \|\sum_{i=1}^d a_i \otimes s_i \right \|_{B(H \otimes \mathcal{F}_{\epsilon})} &\le \left \|\sum_{i=1}^d a_i \otimes l_i \right \|_{B(H \otimes \mathcal{F}_{\epsilon})} + \left \|\sum_{i=1}^d a_i \otimes l_i^* \right \|_{B(H \otimes \mathcal{F}_{\epsilon})}\\
&=  \left \|\sum_{i=1}^d a_i \otimes l_i \right \|_{B(H \otimes \mathcal{F}_{\epsilon})} +  \left \|\sum_{i=1}^d a_i^* \otimes l_i \right \|_{B(H \otimes \mathcal{F}_{\epsilon})}.
\end{align*}
If $\epsilon_{ij}=0$, then we have $l_i^*l_j=0$ ($i\neq j$) or $l_i^*l_i=I$. Thus we have
\begin{align*}
  \left(\sum_{i=1}^d a_i^* \otimes l_i^* \right)\left( \sum_{j=1}^d a_j \otimes l_j\right) &= \sum_{i=1}^d a_i^*a_i \otimes l_i^*l_i + \sum_{i,j=1}^d \epsilon_{ij}a_i^*a_j \otimes l_i^*l_j \\
  &= L^* (\epsilon +I_d) L.
\end{align*}
where $L={}^t(a_1\otimes l_1,\ldots a_d\otimes l_d)$.  Since $\epsilon + I_d\le (\lambda_1+1)I_d$, we have 
\begin{align*}
 L^* (\epsilon +I_d) L  
 &\le  
 (\lambda_1+1)L^*L \\
 &= (\lambda_1+1) \sum_{i=1}^d a_i^*a_i \otimes I. 
 \end{align*}
 By taking operator norms on both sides, we have
 \begin{align*}
  \left \|\sum_{i=1}^d a_i \otimes l_i \right \|^2_{B(H \otimes \mathcal{F}_{\epsilon})} &\le (\lambda_1+1) \left\|\sum_{i=1}^d a_i^*a_i \otimes I\right\|_{B(H\otimes \mathcal{F}_{\epsilon})} \\
  &= (\lambda_1+1) \left\|\sum_{i=1}^d a_i^*a_i \right\|_{B(H)}
  \end{align*}
  By replacing $a_i$ by $a_i^*$, we also have
  \[ \left \|\sum_{i=1}^d a_i^* \otimes l_i \right \|_{B(H \otimes \mathcal{F}_{\epsilon})}^2\le (\lambda_1+1) \left\|\sum_{i=1}^d a_ia_i^* \right\|_{B(H)}. \]
  Moreover, when $\epsilon$ is a connected regular graph with a degree less than $d-1$, we consider the spectral decomposition $\epsilon +I_d = \sum_{i=1}^d (\lambda_i+1) P_i$ where $\lambda_1\ge\lambda_2\ge \cdots \ge \lambda_d$ are the eigenvalues of $\epsilon$ and  $P_i$ is the spectral projection with respect to $\lambda_i$. Since $\epsilon$ is a connected regular graph,  $P_1$ is the matrix whose entries are all $\frac{1}{d}$, and we have $L^*P_1L =\frac{1}{d} (\sum_{i=1}^d a_i^* \otimes l_i^*)( \sum_{j=1}^d a_j \otimes l_j)$
  \begin{align*}
   \sum_{i=1}^d (\lambda_i+1) L^* P_i L &= (\lambda_1+1) L^* P_1 L +\sum_{i=2}^d (\lambda_i+1) L^* P_i L\\
   &\le (\lambda_1+1) L^* P_1 L + (\lambda_2+1) L^*\left(\sum_{i=2}^d P_i\right) L\\
   &=(\lambda_1+1) L^* P_1 L + (\lambda_2+1) L^*(I_d-P_1) L\\
   &=(\lambda_1-\lambda_2) L^* P_1 L + (\lambda_2+1)\sum_{i=1}^d a_i^*a_i \otimes I \\
   &= \frac{\lambda_1-\lambda_2}{d}\left(\sum_{i=1}^d a_i^* \otimes l_i^*\right)\left( \sum_{j=1}^d a_j \otimes l_j\right)   + (\lambda_2+1)\sum_{i=1}^d a_i^*a_i \otimes I
 \end{align*}
By subtracting $(\sum_{i=1}^d a_i^* \otimes l_i^*)( \sum_{j=1}^d a_j \otimes l_j)$, we have 
\[(1-\frac{\lambda_1-\lambda_2}{d})\left(\sum_{i=1}^d a_i^* \otimes l_i^*\right)\left( \sum_{j=1}^d a_j \otimes l_j\right)\le (\lambda_2+1)\sum_{i=1}^d a_i^*a_i \otimes I,\]
and therefore we have
\[ \left \|\sum_{i=1}^d a_i \otimes l_i \right \|^2_{B(H \otimes \mathcal{F}_{\epsilon})} \le \frac{d(\lambda_2+1)}{d-(\lambda_1-\lambda_2)}\left\|\sum_{i=1}^d a_i^*a_i \right\|_{B(H)}. \]
Similarly, by replacing $a_i$ by $a_i^*$, we obtain an inequality for $\sum_{i=1}^da_i^*\otimes l_i$ and the desired inequality. 
\end{proof}
Taking $a_i=1$ for any $i$, we can deduce the following corollary from the above theorem. 
\begin{cor}\label{EpsilonConvolution}
Let $s_1,\dots,s_d$ be $\epsilon$-semicircles. Then we have
\[\left\|\sum_{i=1}^d s_i\right\| \le 2 \sqrt{d (\lambda_1+1)}\]
where $\lambda_1$ is the largest eigenvalue of $\epsilon$.
Moreover, if we assume that $\epsilon$ is a connected regular graph with a degree less than $d-1$, then we also have
$$\left\|\sum_{i=1}^d s_i\right\| \le 2 d \sqrt{\frac{\lambda_2+1}{d-(\lambda_1-\lambda_2)}} $$
where $\lambda_2$ is the second largest eigenvalue of the adjacency matrix of $\epsilon$.
\end{cor}

\begin{rem}\label{OptimalCase}
The first inequality in Corollary \ref{EpsilonConvolution} is optimal in free ($\epsilon=0$) and classical independent ($\epsilon=E_d-I_d$ where all entries of $E_d$ are equal to $1$) case. The second inequality is better than the first because we have
\[d(\lambda_1+1)(1-\frac{\lambda_1-\lambda_2}{d})- d(\lambda_2+1) =(\lambda_1-\lambda_2)(d-\lambda_1-1)\ge 0.\]
Moreover, the second inequality is optimal in the case where one considers a multipartite complete regular graph like below
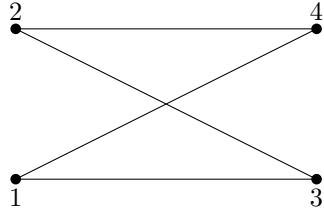
\begin{center}
 \begin{tikzpicture}
\coordinate (1) at (0,0) node [below] at (1) {1};
\coordinate (2) at (0,2) node [above] at (2) {2};
\coordinate (3) at (4,0) node [below] at (3) {3};
\coordinate (4) at (4,2) node [above] at (4) {4};
\draw (1)--(3);
\draw (1)--(4);
\draw (2)--(3);
\draw (2)--(4);
\foreach \x in {1,2,3,4}
\fill[black] (\x) circle (2pt);
\end{tikzpicture}
  \end{center}
In this case, the operator norm of the sum is $4\sqrt{2}$ because $s_1+s_2$ and $s_3+s_4$ are classically independent and each distribution is a semi-circle distribution with variance $2$ (i.e. $\|s_1+s_2\|=2\sqrt{2}$). On the other hand, the largest eigenvalue of the adjacency matrix is $\lambda_1=2$ and the second largest eigenvalue is $\lambda_2=0$. Note that the first inequality yields the bound $2\sqrt{12}=4\sqrt{3}$, but the second inequality gives us the optimal bound $4\sqrt{2}$. This correspondence still holds for $n$-partite complete regular graph with $m$ discrete components ($m=n=2$ in the above case). 
\end{rem}

\begin{rem}
    The norm bounds that we obtained are improved in the paper \cite[Theorem 1.5]{santos2025khintchineinequalitiestracemonoids} with the constant $2\sqrt{\omega(\epsilon)}$ where $w(\epsilon)$ is the largest size of a clique in $\epsilon$. For the proof, they use the operator inequality $\sum_{i=1}^dl_il_i^*\le \omega(\epsilon)I$. Indeed, by combining this inequality with our proof of Theorem \ref{Main}, we obtain a new bound instead of Corollary \ref{EpsilonConvolution}. In the proof of Theorem \ref{Main}, we can replace $L=^t(l_1,\ldots,l_d)$ with $L'=^t(l_1^*,\ldots,l_d^*)$ by applying $(\epsilon_{ij})$-commutation relation $l_i^*l_j-\epsilon_{ij} l_jl_i^* =\delta_{ij} I$. Then, by following the same way as the proof of Theorem \ref{Main} with the inequality $\sum_{i=1}^d l_il_i^*\le \omega(\epsilon)I$, we obtain
    \begin{align*}
        \left(\sum_{i=1}^{d} l_i\right)^*\left(\sum_{i=1}^{d} l_i\right)&\le d + \lambda_1 L'^*L'\\
        &=d+ \lambda_1\sum_{i=1}^d l_il_i^*\\
        &\le d+ \lambda_1\omega(\epsilon).
    \end{align*}
    Therefore, we have the following norm-bound.
\end{rem}
\begin{cor}\label{newinequality}
Let $s_1,\ldots,s_d$ be $\epsilon$-free semicircular elements. Then we have
     \[\left\|\sum_{i=1}^d s_i\right\| \le 2 \sqrt{ d+ \lambda_1\omega(\epsilon)}\]
     where $\lambda_1$ is the largest eigenvalue of  $\epsilon$ and $\omega(\epsilon)$ is the largest size of a clique in $\epsilon$.
\end{cor}
By Wilf's inequality (\cite{WILF1986113}) $\frac{d}{d-\lambda_1}\le w(\epsilon)$, the upper bound $2\sqrt{d+\lambda_1w(\epsilon)}$ is better than $2\sqrt{d w(\epsilon)}$ obtained in \cite[Theorem 1.5]{santos2025khintchineinequalitiestracemonoids} when all operator coefficients are equal to $1$.
\begin{rem}
Based on the proof above, we also have a similar inequality for $\epsilon$-free Haar unitaries $u_1,\ldots,u_d$. These unitaries are closely related to the right-angled Artin group $A(\epsilon)$ associated with $\epsilon$ defined by 
    \[A(\epsilon)=\left\langle g_1,...,g_d\ | \ g_ig_j=g_jg_i, \ (i,j)\in E(\epsilon)\right\rangle\]
    where $E(\epsilon)$ is the set of edges in $\epsilon$. Actually, it is known that $\epsilon$-free Haar unitaries $u_1,\ldots,u_d$ have the same joint distribution as $\lambda(g_1),\ldots,\lambda(g_d)$ with the canonical trace under the left regular representation of $A(\epsilon)$ (See \cite[Prop. 4.2]{MR3552222}). We set $g_{d+i}=g_i^{-1}$ for $i=1,\ldots,d$ and want to estimate $\|\sum_{i=1}^{2d}a_i\otimes \lambda(g_i)\|$.
    In this case, a similar argument works by replacing creation and annihilation operators $l_i,l_i^*$ with operators $u_i,v_i$ defined in \cite{MR1240608} as
\begin{align*}
u_i&=e_i^+ \lambda(g_i), \quad u_{d+i}=e_i^- \lambda(g_i^{-1}) \\
v_{i}&= \lambda(g_i)e_i^-, \quad v_{d+i}= \lambda(g_i^{-1})e_i^+ 
\end{align*}
for $i=1,...,d$ where $e_i^+$ (resp. $e_i^-$) is an orthogonal projection onto the subspace generated by reduced elements in $A(\epsilon)$ starting at $g_i$ (resp. $g_i^{-1}$) up to equivalence from commutation rules. Then, we have $\lambda(g_i)=u_i+v_i$ for $i=1,\ldots,2d$. 
In the same way as the semicircle case, we estimate the norm of $\sum_{i=1}^{2d}a_i \otimes u_i$. 
Note that we have $e_i^{\pm}e_i^{\mp}=0$ and $e_i^{\pm}e_j^{\pm}=0=e_i^{\mp}e_j^{\pm}$ if $\epsilon_{ij}=0$ and $i\neq j$. This implies the following identity 
\[\sum_{i=1}^{2d} a_i^*\otimes u_i^*\sum_{i=1}^{2d} a_j\otimes u_j=U^*\begin{pmatrix}
    I_d+\epsilon & \epsilon\\ \epsilon & I_d+\epsilon
\end{pmatrix}U \]
where $U={}^t(u_1,\ldots,u_{2d}) $ Then, we use the spectral decomposition $\epsilon=\sum_{i=1}^d\lambda_i P_i$ of $\epsilon$ to obtain 
\begin{align*}
U^*
    \begin{pmatrix}
    I_d+\epsilon & \epsilon\\ \epsilon & I_d+\epsilon
\end{pmatrix} U &= \begin{pmatrix}
    I_d & 0 \\ 0 & I_d
\end{pmatrix}+ \sum_{i=1}^d \lambda_i U^*\begin{pmatrix}
    P_i & P_i \\ P_i & P_i
\end{pmatrix} U\\
&\le U^*\begin{pmatrix}
    (\lambda_1+1)I_d & \lambda_1I_d\\\lambda_1I_d & (\lambda_1+1)I_d
\end{pmatrix}U. \\
\end{align*}
By using the identity $u_i^*u_{d+i}=0$ ($i=1,\ldots,d$), we have
\begin{align*}
U^*\begin{pmatrix}
    (\lambda_1+1) I_d & \lambda_1 I_d\\ \lambda_1 I_d & (\lambda_1+1)I_d
\end{pmatrix} U&=(\lambda_1+1)U^*\begin{pmatrix}
    I_d & 0 \\ 0 & I_d
\end{pmatrix} U\\
&= (\lambda_1+1)\sum_{i=1}^{2d} a_i^*a_i \otimes u_i^*u_i\\
&\le (\lambda_1+1)\sum_{i=1}^{2d} a_i^*a_i \otimes I .
\end{align*}
where we use the fact that each $u_i^*u_i$ is an orthogonal projection. Thus, by taking the operator norm, we get 
\[\left\|\sum_{i=1}^{2d} a_i \otimes u_i\right\|_{B(H\otimes l^2(A(\epsilon)))}^2\le (\lambda_1+1) \left\|\sum_{i=1}^{2d}a_i^*a_i\right\|. \]
In the same way, we also have the second estimate in Theorem \ref{Main} for $\epsilon$-free Haar unitaries. To obtain the analog of the inequality in Corollary \ref{newinequality}, we use the $(\epsilon_{ij})$-commutation relation $u_i^*u_j=\epsilon_{ij} u_ju_i^*$ for $i\neq j$. 
\[ \sum_{i=1}^{2d}u_i^*\sum_{i=1}^{2d}u_i =\sum_{i=1}^{2d} u_i^*u_i+U'^*\begin{pmatrix}
    \epsilon & \epsilon\\ \epsilon & \epsilon
\end{pmatrix} U'\]
where $U'={}^t(u_1^*,\ldots,u_{2d}^*)$. Since $2\lambda_1$ is the largest eigenvalue of the matrix $\epsilon \otimes E_2$, we have $\epsilon \otimes E_2 \le 2\lambda_1$ and
\begin{align*}
    \sum_{i=1}^{2d}u_i^*\sum_{i=1}^{2d}u_i &\le\sum_{i=1}^{2d} u_i^*u_i+2\lambda_1U'^*U' \\
    &=\sum_{i=1}^{2d} u_i^*u_i+2\lambda_1\sum_{i=1}^{2d} u_iu_i^*\\
    &\le 2d+2\lambda_1 w(\epsilon)
\end{align*}
where we use $u_i^*u_i\le 1$ and $\sum_{i=1}^{2d}u_iu_i^*\le w(\epsilon)$ (see the proof of Theorem 1.5 and Remark 5.1 in \cite{santos2025khintchineinequalitiestracemonoids}). Therefore, we have the Haar unitary version of Corollary \ref{newinequality} 
\[\left\|\sum_{i=1}^{2d}\lambda(g_i)\right\|\le 2\sqrt{2d + 2\lambda_1w(\epsilon)}.\]
\end{rem}

Next, we turn to lower bounds.

\begin{thm}
Let $s_1,\ldots,s_d$ be $\epsilon$-free semicircular elements. Then we have
\[
\left\|\sum_{i=1}^d s_i\right\| \ge \max \left\{\sqrt{4\omega(\epsilon)^2 + d - \omega(\epsilon)},\ 2\sqrt{d} \right\}
\]
where $\omega(\epsilon)$ is the largest size of a clique in $\epsilon$. 
\end{thm}
\begin{proof}
By using the identity of $\mathrm{C}^*$-probability space (e.g. \cite[Proposition 3.17]{MR2266879}), we have \[\left\|\sum_{i=1}^d s_i\right\|=\lim_{n\to \infty}\tau\left[\left(\sum_{i=1}^ds_i\right)^{2n}\right]^{\frac{1}{2n}}.\] 
Since each term $\tau\left[s_{i(1)}s_{i(2)}\cdots s_{i(2n)}\right]$ in the expansion of $\tau\left[\left(\sum_{i=1}^ds_i\right)^{2n}\right]$ is a sum over $(\epsilon,i)$-non-crossing (pair) partitions (\cite[Definition 5.1]{MR3552222}) which contains non-crossing (pair) partitions, the free case ($\epsilon=0$) attains minimal value of $\left\|\sum_{i=1}^d s_i\right\|$. Since the operator norm of the sum of $d$ free semicircles is equal to $2\sqrt{d}$ (e.g. \cite[Exercise 7.27]{MR2266879}), we have    
\[ \left\|\sum_{i=1}^d s_i\right\| \ge 2 \sqrt{d}\]
for any $\epsilon$-free semicircular elements.

On the other hand,
let $A=\{i_1,\ldots,i_k\}$ be a clique of $\epsilon$. Then for $N \in \mathbb{N}$, we consider the vector in $\mathcal{F}_{\epsilon}$
$$\xi_N=\sum_{j_1,\ldots,j_k=1}^N e_{i_1^{j_1}\cdots i_k^{j_k}}.$$
Since all summands are mutually orthogonal, the square of the norm of this vector is equal to $N^k$. Since $\{i_1,\ldots,i_k\}$ is a clique, we have
\begin{eqnarray*}
s_{i_m}\xi_N &=& \sum_{j_1,\ldots,j_k=1}^N e_{i_1^{j_1}\cdots i_m^{j_m+1}\cdots i_k^{j_k}} + \sum_{j_1,\ldots,j_k=1}^N e_{i_1^{j_1}\cdots i_m^{j_m-1}\cdots i_k^{j_k}}\\
&=& \sum_{\substack{1\le j_n \le N, \ n \neq m \\ 2 \le j_m \le N-1}} 2 e_{i_1^{j_1}\cdots i_k^{j_k}}+ \sum_{\substack{1\le j_n \le N, \ n \neq m \\ j_m=0,1,N,N+1}} e_{i_1^{j_1}\cdots i_k^{j_k}}. 
\end{eqnarray*}
By taking summation over $m$, we have
\begin{eqnarray*}
\sum_{m=1}^k s_{i_m} \xi_N = \eta + \sum_{j_1,\ldots,j_k=2}^{N-1} 2k e_{i_1^{j_1}\cdots i_k^{j_k}} 
\end{eqnarray*}
where $\eta$ is orthogonal to the second term and $\|\eta\|^2 = O(N^{k-1})$. On the other hand, if $i \not \in A$, we have
$$s_i \xi_N = \sum_{j_1,\ldots,j_k=1}^N e_{i \cdot i_1^{j_1}\cdots i_k^{j_k}}$$
which is orthogonal to $\sum_{m=1}^k s_{i_m} \xi_N$.
Thus we have
\[\left\|\sum_{i=1}^d s_i \xi_N\right\|^2=4k^2(N-2)^k+(d-k)N^k + O(N^{k-1}),\]
and we obtain
$$\left\|\sum_{i=1}^d s_i\right\| \ge \lim_{N\to \infty} \frac{\|\sum_{i=1}^d s_i \xi_N\| }{\|\xi_N\|} = \sqrt{4k^2+d-k}. $$
We obtain the theorem by taking $k=\omega(\epsilon)$.

\end{proof}

\begin{exam}
Let us consider the XY-model (complement graphs of line graphs, i.e. only neighboring elements are free). We assume that $s_1$ and $s_d$ are free, so that we have a regular graph. This is a connected $(d-3)$-regular graph, and the second largest eigenvalue is close to $1$ as $d \to \infty$. 

To see this, note that the eigenvalues of a cyclic permutation matrix $P\in M_d(\mathbb{C})$ ($Pe_i=e_{i+1}$ mod $d$ for a canonical basis of $\mathbb{C}^d$) are $\exp\left[-\frac{2\pi k}{d}\sqrt{-1}\right]$ ($k=0,\ldots,d-1$) with a corresponding eigenvector $v_k={}^t\left(1,\exp\left[\frac{2\pi k}{d}\sqrt{-1}\right],\ldots,\exp\left[\frac{2\pi k(d-1)}{d}\sqrt{-1}\right]\right)$, and $E_d-I_d-(P+P^*)$ is what we want to compute.
Thus, when $d$ is odd, we have $\lambda_1=d-3$ with the eigenvector $v_0$, $\lambda_{2k}=\lambda_{2k+1}=-1-2\cos(\frac{2\pi}{d}(\frac{d+1}{2}-k))$ with the eigenvectors $v_{\frac{d+1}{2}-k},v_{\frac{d-1}{2}+k}$ for $k=1,\ldots,\frac{d-1}{2}$. When $d$ is even, we have $\lambda_1=d-3$ with the eigenvector $v_0$, $\lambda_2=1$ with the eigenvector $v_{\frac{d}{2}}$, and $\lambda_{2k+1}=\lambda_{2k+2}=-1-2\cos(\frac{2\pi}{d}(\frac{d}{2}-k))$ with the eingenvectors $v_{\frac{d}{2}-k},v_{\frac{d}{2}+k}$ for $k=1,\ldots,\frac{d}{2}-1$. 
When $d$ is odd, Corollary \ref{EpsilonConvolution} gives the bound:
\begin{align*}
 \left\|\sum_{i=1}^d s_i\right\| &\le 2d \sqrt{-2\cos\left(\frac{\pi(d-1)}{d}\right)\left(d-(d-3+1+2\cos\left(\frac{\pi(d-1)}{d}\right)\right)^{-1}}\\
 & = 2d \sqrt{-\cos\left(\frac{\pi(d-1)}{d}\right)\left(1-\cos\left(\frac{\pi(d-1)}{d}\right)\right)^{-1}}.
\end{align*}
When $d$ is even, Corollary \ref{EpsilonConvolution} gives a bound $2d\sqrt{\frac{1+1}{d-(d-3-1)}}=\sqrt{2}d$. On the other hand, by triangular inequality, we also have
$$\left\|\sum_{i=1}^d s_i\right\| \le \|s_1+s_2\| \cdot \frac{d}{2}=2\sqrt{2}\cdot \frac{d}{2}= \sqrt{2} d.$$
Regarding lower bounds for the XY-model,
we can take $\{2k-1\}_{k=1}^{\lceil d/2\rceil}$ as a maximal clique and we apply Theorem \ref{CliqueEstimate} to have
$$\left\|\sum_{i=1}^d s_i\right\| \ge \sqrt{4\lceil d/2 \rceil^2+d-\lceil d/2 \rceil}> 2 \lceil d/2 \rceil.$$
Note that we can easily see $2 \lceil d/2 \rceil$ is a lower bound since $\{s_{2k-1}\}_{k=1}^{\lceil d/2\rceil}$ is a classically independent family of semicircles, so our result is a non-trivial improvement on naive lower bounds.
\end{exam}

\bibliographystyle{amsalpha}
\bibliography{reference.bib}
    
\end{document}